\documentclass[11pt]{amsart}

\usepackage{amsmath, amssymb, amsthm,mathtools,tikz,tikz-cd,hyperref,enumerate}
\usepackage[shortlabels]{enumitem}

\usepackage[utf8]{inputenc}

\usepackage[english]{babel}

\hypersetup{
    colorlinks,
    citecolor=blue,
    filecolor=blue,
    linkcolor=blue,
    urlcolor=blue
}

\newcommand{\C}{{\mathbb C}}

\newcommand{\Hh}{\mathbb{H}}

\newcommand{\p}{\mathfrak{p}}

\newcommand{\End}{\operatorname{End}}

\newcommand{\Ind}{\operatorname{Ind}}

\newcommand{\inv}{\operatorname{inv}}
\newcommand{\Br}{\operatorname{Br}}

\newcommand{\JL}{\mathrm{JL}}

\newcommand{\R}{{\mathbb R}}

\newcommand{\Z}{{\mathbb Z}}
\newcommand{\F}{{\mathbb F}}

\newcommand{\cL}{{\mathcal L}}

\newcommand{\Q}{{\mathbb Q}}
\newcommand{\cO}{{\mathfrak o}}

\newcommand{\cusp}{\mathrm{cusp}}

\newcommand{\SL}{{\mathrm{SL}}}

\newcommand{\GL}{{\mathrm{GL}}}

\newcommand{\PGL}{{\mathrm{PGL}}}
\newcommand{\Aut}{{\mathrm{Aut}}}

\newcommand{\Cent}{{\mathrm{Z}}}

\DeclareMathOperator{\Gal}{Gal}

\numberwithin{equation}{subsection}
\newtheorem{thm}[equation]{Theorem}

\newtheorem{lemma}[equation]{Lemma}
\newtheorem{conjecture}[equation]{Conjecture}
\newtheorem{cor}[equation]{Corollary}

\numberwithin{equation}{subsection}
\newtheorem{prop}[equation]{Proposition}

\newtheorem{hypo}[equation]{Hypothesis}

\theoremstyle{definition}
\numberwithin{equation}{subsection}
\newtheorem{remark}[equation]{Remark}

\newtheorem{defn}[equation]{Definition}

\newtheorem{prop-def}{Proposition-Definition}

\makeatletter
\def\@settitle{\begin{center}
    \normalfont
\uppercasenonmath\@title
  \@title
  \end{center}
}

\title[Rationality of Jacquet-Langlands]{Rationality of the local Jacquet-Langlands Corresponence for $\GL_n$}

\author{Kenta Suzuki}
\address{M.I.T., 77 Massachusetts Avenue,
Cambridge, MA, USA}
\email{kjsuzuki@mit.edu}

\begin{document}

\maketitle

\begin{abstract}
We relate the field of definition of representations $\sigma$ of the group of units $D^\times$ of a non-archimedean division algebra $D/F$ to that of its L-parameter $\varphi_\sigma\colon W_F\to \GL_n(\C)$, extending results of \cite{Prasad-Ramakrishnan}. The field of definitions are controlled by division algebras $\mathcal D_{\sigma}$ and $\mathcal D_{\varphi_\sigma}$ over the field of rationality $\Q(\pi)$, and we completely pin down the relationship between the Hasse invariants at places not over $p$. Under some additional assumptions we can also specify the Hasse invariants at places over $p$.
\end{abstract}

\section{Introduction}

Let $F/\Q_p$ be a non-archimedean local field and let $D/F$ be a central division algebra of dimension $n^2$. The local Langlands correspondence and the Jacquet-Langlands correspondence provide a bijection between:
\begin{itemize}
\item irreducible square-integrable representations of $\GL_n(F)$ with central character of finite order;
\item irreducible smooth representations of $D^\times$ with central character of finite order; and
\item semisimple $n$-dimensional representations of the absolute Galois group $\Gamma_F$.
\end{itemize}
For each representation $\pi$ of $\GL_n(F)$, we let $\JL(\pi)$ denote the corresponding representation of $D^\times$ and let $\cL_F(\pi):=\varphi_\pi(-\frac{n-1}2)$ denote the corresponding \emph{twisted} L-parameter. Prasad and Ramakrishnan \cite{Prasad-Ramakrishnan} then relates whether or not $\JL(\pi)$ descends to a real representation to whether $\cL_F(\pi)$ descends to a real representation.

We extend their result, and given a number field $K\subset\C$, we relate whether or not $\JL(\pi)$ descends to a $K$-representation to whether or not $\cL_F(\pi)$ descends to a $K$-representation. Clearly if representations descend to $K$ then $K\supseteq\Q(\pi)$, the \emph{field of rationality} of $\pi$ (see Definition~\ref{defn:field-of-rationality}).

Such rationality questions are controlled by central division algebras $\mathcal D_{\JL(\pi)}$ and $\mathcal D_{\cL_F(\pi)}$ over $\Q(\pi)$, which are common generalizations of the Frobenius-Schur indicator dealt with in \cite{Prasad-Ramakrishnan}, and the Schur index (see the discussion after Definition~\ref{defn:division-algebra}). We completely pin down the relationship between the two division algebras in places away from $p$:
\begin{thm}
Let $\pi$ be an irreducible supercuspidal representation of $\GL_n(F)$, such that $\omega_\pi$ is of finite order. Then, for a place $v\nmid p$ of $\Q(\pi)$,
\[
\inv_v(\mathcal D_{\JL(\pi)})+\inv_v(\mathcal D_{\cL_F(\pi)})=\begin{cases}
0&v\nmid p,\infty\\
\frac1{[\C:\R(\pi)]}&v|\infty.
\end{cases}
\]
\end{thm}

As a consequence, we partially resolve \cite[Conjecture~8.1]{Prasad-Ramakrishnan}. For example, we prove:
\begin{thm}
Let $\pi$ be an irreducible self-dual supercuspidal representation of $\GL_n(F)$. Then for any place $v$ of $\Q(\pi)$,
\[
\inv_v(\mathcal D_{\JL(\pi)}),\inv_v(\mathcal D_{\cL_F(\pi)})\in\frac12\Z/\Z.
\]
In particular, when $\Q(\pi)$ has an odd number of places above $p$ (in particular when $[\Q(\pi):\Q]$ is odd), for any $v|p$,
\begin{equation}
\inv_v(\mathcal D_{\JL(\pi)})+\inv_v(\mathcal D_{\cL_F(\pi)})=\frac12[\Q(\pi):\Q].
\end{equation}
\end{thm}

\section{Acknowledgement}
The author thanks Guy Henniart and Yoichi Mieda for helpful discussions.

\section{Preliminaries on rationality}

\begin{defn}\label{defn:field-of-rationality}
Let $G$ be a locally pro-finite group and let $(\pi,V)$ be a complex representation of $G$. The \emph{field of rationality} of $\pi$ is $\Q(\pi):=\{z\in\C:\gamma\cdot z=z,\forall \gamma\in\Gamma\}$, where $\Gamma(\pi)\subseteq\Aut(\C/\Q)$ is the group of $\gamma\in\Aut(\C/\Q)$ such that $\pi^\gamma\cong\pi$. Moreover, for each field $K\subset\C$, let $K(\pi)=K.\Q(\pi)$.
\end{defn}

\begin{remark}\label{galois-remark}
When $\pi$ is an irreducible representation of a finite group $G$, the field of rationality $\Q(\pi)$ is always Galois. Indeed, it is a sub-field of $\Q(\zeta_{|G|})$. However, in general this need not be the case: consider $\Z\to\C^\times:1\mapsto\sqrt[3]2$.
\end{remark}

\begin{defn}
Let $G$ be a group and let $(\pi,V)$ be a complex representation of $G$. The representation $\pi$ is \emph{defined over a sub-field $K\subset\C$} if there exists a $K$-vector space representation $(\pi_0,V_0)$ such that $V=V_0\otimes_K\C$. 
\end{defn}

\begin{remark}
If $(\pi,V)$ is defined over a field $K$, then $K$ contains the field of rationality $\Q(\pi)$, since certainly $\Aut(\C/K)\subseteq\Gamma(\pi)$. However, in general $(\pi,V)$ need not be defined over $\Q(\pi)$. Indeed, the quaternion group $Q_8$ has a complex $2$-dimensional representation whose character is rational, but the representation \emph{is not} defined over $\Q$.
\end{remark}

\begin{remark}\label{non-unique-remark}
There need not be a minimal field $K$ such that a representation $\pi$ is defined over $K$. The $2$-dimensional representation of $Q_8$ is not defined over $\Q$, but can be defined over any quadratic field $K/\Q$ such that $a^2+b^2+1=0$ has a solution:
\[
i\mapsto \begin{pmatrix}a&b\\b&-a\end{pmatrix},j\mapsto \begin{pmatrix}b&-a\\-a&-b\end{pmatrix},k\mapsto \begin{pmatrix}0&1\\-1&0\end{pmatrix}.
\]
\end{remark}

To each representation, we can attach an auxilary $\Q$-algebra, often convenient in addresssing rationality questions:
\begin{defn}\label{defn:division-algebra}
For a (possibly reducible) representation $(\pi,V)$ let $\Q\{\pi\}\subset\End_\C(V)$ be the $\Q$-span of $\{\pi(g):g\in G\}$. It inherits the natural ring structure from $\End_\C(V)$.
\end{defn}
For convenience, let us assume $\Q\{\pi\}$ is finite-dimensional over $\Q$, which implies $\Q\{\pi\}$ is simple, i.e., of the form $M_n(\mathcal D_\pi)$ for some division algebra $\mathcal D_\pi$. Then:

\begin{lemma}
Let $(\pi,V)$ be an irreducible representation of a locally pro-finite group $G$, such that $\Q\{\pi\}$ is finite-dimensional over $\Q$. Then $\Q(\pi)=\Cent(\Q\{\pi\})=\Cent(\mathcal D_\pi)$. Moreover, $\pi$ is defined over a field $K$ if and only if $K\otimes_{\Q(\pi)}\mathcal D_\pi$ (equivalently, $K\otimes_{\Q(\pi)}\Q\{\pi\}$) is split, i.e., isomorphic to $M_k(K)$ for some integer $k$.
\end{lemma}
\begin{proof}
For $\sigma\in\Aut(\C/\Q)$, clearly $\pi\cong\pi^\sigma$ implies that for each $z\in\Cent(\Q\{\pi\})$ the elements $z,\sigma(z)\in\End_\C(V)$ are conjugate to each other, which, since $z$ is central, shows $z=\sigma(z)$. Thus, $\Gamma_\pi\subseteq\Aut(\C/\Q(\pi))$, in the notation of Definition~\ref{defn:field-of-rationality}. Conversely, since $\Q\{\pi\}$ is a central simple algebra over $\Cent(\Q\{\pi\})$, so is $\Q\{\pi\}\otimes_{\Cent(\Q\{\pi\})}\C$. Thus, $\Q\{\pi\}\hookrightarrow\End_\C(V)$ induces an isomorphism $\Q\{\pi\}\otimes_{\Cent(\Q\{\pi\})}\C\cong\End_\C(V)$. Thus, the $\Q\{\pi\}$-representation $V$ is simply the canonical representation, which is in particular invariant under the action of $\Aut(\C/\Cent(\Q\{\pi\}))$. Thus $\Aut(\C/\Cent(\Q\{\pi\}))\subset\Gamma_\pi$.

Now, $\pi$ is defined over $K/\Q(\pi)$ if and only if there is some $K$-vector space $V_0$ such that $V_0\otimes_K\C=V$ and $\Q\{\pi\}$ has image in $\End_K(V_0)$. This is exactly equivalent to $\Q\{\pi\}\otimes_{\Q(\pi)}K$ splitting.
\end{proof}

\begin{remark}
The equality $\Q(\pi)=\Cent(\Q\{\pi\})$ need not hold when $\pi$ is reducible. For example, let $\pi$ be the representation of the cyclic group $C_4$ given by $\begin{pmatrix}&1\\-1\end{pmatrix}\in\GL_2(\Q)$. Then $\Q(\pi)=\Q$, while $\Q\{\pi\}=\Q(\sqrt{-1})$.
\end{remark}

\begin{remark}
The index of $\mathcal D_\pi$, the square root of its $\Q(\pi)$-dimension, is usually called the \emph{Schur index} of $\pi$.
\end{remark}

\begin{remark}
We can now re-phrase the observation in Remark~\ref{non-unique-remark} in the language of division algebras. For the $2$-dimensional representation $\pi$ of $Q_8$ we have $\Q\{\pi\}=\Hh_\Q=(-1,-1)_\Q$, a quaternion algebra over $\Q$. Now, $\pi$ is defined over $K$ if and only if $(-1,-1)_K\cong M_2(K)$.

The algebra $[\Hh_\Q]\in\Br(\Q)$ is such that $\inv_2([\Hh_\Q])=\frac12$ and $\inv_\infty([\Hh_\Q])=\frac12$, and $\inv_p([\Hh_\Q])=0$ for all $p\ne 2,\infty$. Thus, $[\Hh_\Q\otimes_\Q K]=0\in\Br(K)$ if and only if the quadratic extension $K/\Q$ is inert over $2$ and $\infty$, i.e., $K$ is of the form $\Q(\sqrt d)$ where $d<0$ is a square-free integer such that $d\equiv2,3\pmod4$.
\end{remark}

\begin{lemma}\label{mult-one-lemma}
Let $\pi$ be an irreducible $\C$-representation of a locally pro-finite group $G$ which appears with multiplicity one in a $K$-rational representation $\Pi$ for some $K\subset\C$. Then $\pi$ is defined over $K(\pi)$.
\end{lemma}
\begin{proof}
View $\pi\subseteq\Pi$ as a subspace. Then for each $\sigma\in\Aut(\C/K(\pi))$, by multiplicity one, we have $\pi=\pi^\sigma$ as subspaces of $\Pi$. Thus, $\pi$ is defined over $K(\pi)$.
\end{proof}

\section{The field of definition for supercuspidal representations of $\GL_n(F)$}

Let $q$ be the order of the residue field $\cO_F/\p_F$ of $F$. For reasons in Remark~\ref{galois-remark}, henceforth we assume:

\begin{hypo}\label{finite-order-hypo}
$\pi$ is an irreducible representation of $\GL_n(F)$ such that one of the following equivalent conditions hold:
\begin{itemize}
\item the central character $\omega_\pi$ is of finite order
\item $\JL(\pi)$ factors through a finite quotient of $D^\times$
\item the L-parameter $\varphi_\pi\colon W_F\times\SL_2(\C)\to\GL_n(\C)$ extends to $\Gamma_F\times\SL_2(\C)$.
\end{itemize}
\end{hypo}

The equivalence follows from:
\begin{lemma}
Let $(\varphi,V)$ be an irreducible representation of $W_F$. Then:
\begin{enumerate}
\item\label{item1} for any $g\in W_F$, there exists an integer $N>0$ such that $\varphi(g)^N$ acts as a scalar on $V$. In other words, $\varphi\colon W_F\to\PGL(V)$ factors through $\Gamma_F$.
\item\label{item2} there exists an unramified character $\chi$ such that $\chi^{-1}\varphi$ extends to a representations of $\Gamma_F$.
\end{enumerate}
\end{lemma}
\begin{proof}
Clearly \eqref{item2} follows from \eqref{item1}. To prove \eqref{item1}, note that it suffices to check it for a lift $\varpi\in W_F$ of the Frobenius of the residue field $k_F$, since the inertia $I_F$ is pro-finite. Now, the action of $I_F$ factors through some finite subgroup $G\subset\GL(V)$, and $\varpi$ acts as an automorphism of $G$, so there exists an integer $N>0$ such that $\varpi^N$ acts trivially on $G$. Then $\varpi^N$ is central in $G\rtimes\langle\varpi\rangle$, so by Schur's lemma $\varphi(\varpi)^N$ is a scalar.
\end{proof}

\begin{lemma}
The local Langlands correspondence and Jacquet-Langlands correspondence, twisted by $\nu^{(n-1)/2}$ are invariant under $\Aut(\C/\Q)$. That is, for any $\sigma\in \Aut(\C/\Q)$:
\[
\cL_F(\pi^\sigma)=\cL_F(\pi)^\sigma,\JL(\pi^\sigma)=\JL(\pi)^\sigma.
\]
In particular, given an irreducible square-integrable representation $\pi$ of $\GL_n(F)$, we have:
\[
\Q(\pi)=\Q(\JL(\pi))=\Q(\cL_F(\pi)).
\]
\end{lemma}
\begin{proof}
Indeed, the characterizing properties in terms of character formulae by \cite{henniart-characterization} and \cite{DKV} are invariant (after twisting) under $\Aut(\C/\Q)$. For $\GL_n$, this was observed in \cite[Lem~VII.1.6.2]{harris-taylor}.
\end{proof}

First, we can show that for representations of $\GL_n(F)$, there are no obstructions to defining representations over the field of rationality:

\begin{lemma}\label{gl-n-rationality}
Any irreducible square-integrable representation $\pi$ of $\GL_n(F)$ is defined over $\Q(\pi)$.
\end{lemma}
\begin{proof}
Square-integrable representations of $\GL_n(F)$ are generic, so $\pi$ has a unique Whittaker model. That is, $\pi$ appears with multiplicity one in $\Ind_N^G(\psi)$, where $\psi$ is a non-degenerate character of $N$. By mimicking the proof of \cite{gow}, we see that $\Ind_N^G(\psi)$ is defined over $\Q$, and hence by Lemma~\ref{mult-one-lemma} the representation $\pi$ is defined over $\Q(\pi)$.
\end{proof}

We can pin-down the relationship between $\mathcal D_{\JL(\pi)}$ and $\mathcal D_{\cL_F(\pi)}$ in places away from $p$. For each place $v$ of $\Q(\pi)$, let $\inv_v\colon\Br(\Q(\pi))\to\Q/\Z$ be the Hasse invariant map, which fits into the standard short exact sequence
\[
1\to\Br(\Q(\pi))\xrightarrow{\bigoplus_v\inv_v}\bigoplus_{v\nmid\infty}\Q/\Z\oplus\bigoplus_{v|\infty}\frac12\Z/\Z\xrightarrow{\Sigma}\Q/\Z\to0.
\]
\begin{thm}\label{index-calculation}
Let $\pi$ be an irreducible supercuspidal representation of $\GL_n(F)$ satisfying Hypothesis~\ref{finite-order-hypo}. Then, for a place $v\nmid p$ of $\Q(\pi)$,
\[
\inv_v(\mathcal D_{\JL(\pi)})+\inv_v(\mathcal D_{\cL_F(\pi)})=\begin{cases}
0&v\nmid p,\infty\\
\frac1{[\C:\R(\pi)]}&v|\infty.
\end{cases}
\]
Moreover,
\begin{equation}\label{sum-identity}
\sum_{v|p}\inv_v(\mathcal D_{\JL(\pi)})+\inv_v(\mathcal D_{\cL_F(\pi)})=\frac12[\Q(\pi):\Q]\in\Q/\Z.
\end{equation}
\end{thm}
\begin{proof}
It suffices to prove this for a division algebra $D/F$ with $\inv(D)=1/n$, where the Jacquet-Langlands correspondence admits a geometric description. Then for each finite $\ell\ne p$, by \cite{mieda-jl} the \'{e}tale cohomology
\[
H^{n-1}_{\mathrm{LT}}:=\lim_{\substack{\longrightarrow\\ m}}H_c^{n-1}((M_m/\varpi^\Z)\otimes_{\breve{F}}\overline{\breve{F}},\Q_\ell)
\]
carries an action of $\GL_n(F)\times D^\times\times W_F$, and the cuspidal part $H^{n-1}_{\mathrm{LT},\cusp}\otimes_{\Q_\ell}\overline\Q_\ell$ contains $\pi^\vee\boxtimes \JL(\pi)\boxtimes\cL_F(\pi)$ with multiplicity one. Thus by Lemma~\ref{mult-one-lemma} the representation $\pi^\vee\boxtimes \JL(\pi)\boxtimes\cL_F(\pi)$ is defined over the field of rationality $\Q_\ell(\pi)$. Furthermore, since by Lemma~\ref{gl-n-rationality} the representation $\pi$ is defined over $\Q_\ell(\pi)$, and hence the representation $\JL(\pi)\boxtimes\cL_F(\pi)$ of $D^\times\times W_F$ is also defined over $\Q_\ell(\pi)$.

Now, since $\Q\{\JL(\pi)\boxtimes\cL_F(\pi)\}\cong\Q\{\JL(\pi)\}\otimes\Q\{\cL_F(\pi)\}$, we have:
\[
0=\inv_{\Q_\ell(\sqrt q)(\pi)}(\mathcal D_{\JL(\pi)\boxtimes\cL_F(\pi)})=\inv_{\Q_\ell(\pi)}(\mathcal D_{\JL(\pi)})+\inv_{\Q_\ell(\pi)}(\mathcal D_{\cL_F(\pi)}),
\]
i.e.,
\(
\inv_{\Q_\ell(\pi)}(\mathcal D_{\JL(\pi)})=-\inv_{\Q_\ell(\pi)}(\mathcal D_{\cL_F(\pi)}).
\)

Moreover, for places $v$ over infinity, \cite{Prasad-Ramakrishnan} tells us
\[
\inv_{\Q(\pi)_v}(\mathcal D_{\JL(\pi)})+\inv_{\Q(\pi)_v}(\mathcal D_{\cL_F(\pi)})=\frac1{[\C:\R(\pi)]}\in\frac1{[\C:\R(\pi)]}\Z/\Z.
\]
Indeed, when $\R(\pi)=\C$ the statement trivially holds. Moreover, $\R(\pi)=\R$ is equivalent to the existence of a real character $\chi\colon F^\times\to\R_{>0}$ such that $\chi^{-1}\otimes\pi$ is self-dual. Now the equation follows since $\varphi_{\chi^{-1}\pi}$ is symplectic if and only if $\inv_{\R}(\mathcal D_{\varphi_{\chi^{-1}\pi}})=\frac12$.

Now, using the fact that $\sum_{v}\inv_v=0$, we have
\begin{align*}
\sum_{v}\inv_v(\mathcal D_{\JL(\pi)})+\inv_v(\mathcal D_{\cL_F(\pi)})&=\sum_{v|p}\inv_v(\mathcal D_{\JL(\pi)})+\inv_v(\mathcal D_{\cL_F(\pi)})+\sum_{v|\infty}\frac1{[\C:\R(\pi)]}\\
&=\sum_{v|p}\inv_v(\mathcal D_{\JL(\pi)})+\inv_v(\mathcal D_{\cL_F(\pi)})+\frac12[\Q(\pi):\Q]\\
&=0,
\end{align*}
since $\sum_{v|\infty}[\Q(\pi)_v:\R]=\dim_\R(\Q(\pi)\otimes_\Q\R)=[\Q(\pi):\Q]$.
\end{proof}

Thus, the main problem now is to calculate the relation between $\inv_v(\mathcal D_{\JL(\pi)})$ and $\inv_v(\mathcal D_{\cL_F(\pi)|_{W_F}})$ for $v|p$. Recall \cite[Conjecture~8.1]{Prasad-Ramakrishnan}:
\begin{conjecture}\label{main-conjecture}
Let us be in the setting of Theorem~\ref{index-calculation}.
\begin{enumerate}
\item\label{orthogonal-conjecture} If $\JL(\pi)$ is not self-dual, or $\JL(\pi)$ is orthogonal, then for each $v|p$,
\[
\mathcal D_{\JL(\pi)}=\mathcal D_{\cL_F(\pi)}.
\]
\item\label{symplectic-conjecture} If $\JL(\pi)$ is symplectic, then:
\begin{itemize}
\item If $[\Q(\pi):\Q]$ is even, then for each $v<\infty$,
\[
\inv_v(\mathcal D_{\JL(\pi)})=\inv_v(\mathcal D_{\cL_F(\pi)}).
\]
\item If $[\Q(\pi):\Q]$ is odd, then for each $v\nmid p,\infty$,
\[\inv_v(\mathcal D_{\JL(\pi)})=\inv_v(\mathcal D_{\cL_F(\pi)}),\]
and for each $v|p$,
\[
\inv_v(\mathcal D_{\JL(\pi)})\in\inv_v(\mathcal D_{\cL_F(\pi)})+\frac12\Z.
\]
\end{itemize}
\end{enumerate}
\end{conjecture}

\begin{remark}
Although \cite{Prasad-Ramakrishnan} phrases many of their results in terms of $\inv_v(\mathcal D_{\JL(\pi)})-\inv_v(\mathcal D_{\cL_F(\pi)})$, we see that it is \emph{more natural} to ask about the sum $\inv_v(\mathcal D_{\JL(\pi)})+\inv_v(\mathcal D_{\cL_F(\pi)})$.
\end{remark}

However, in some cases, Theorem~\ref{index-calculation} already gives a complete picture. An easy example is:
\begin{cor}\label{easy-cor}
When $\Q(\pi)$ has a unique prime lying over $p$,
\[
\inv_v(\mathcal D_{\JL(\pi)})+\inv_v(\mathcal D_{\cL_F(\pi)})=\begin{cases}
0&v\nmid p,\infty\\
\frac12[\Q(\pi):\Q]&v| p\\
\frac1{[\C:\R(\pi)]}&v|\infty.
\end{cases}
\]
\end{cor}
We also have the following arithmetic result due to \cite{schur-subgroup-i} and \cite{schur-subgroup-ii}:
\begin{lemma}\label{schur-subgroup-lemma}
    Let $K/\Q$ be a finite abelian extension and let $[\mathcal D]\in\Br(K)$ be a central simple algebra arising as a simple factor of the group algebra $K[G]$ of some finite group $G$, and fix a rational prime $p>0$. Then the order of $m=\inv_\p([\mathcal D])$ is independent of the prime $\p|p$ of $K$. Moreover, $K$ contains a primitive $m$-th root of unity $\zeta_m$, and for $\sigma\in\Gal(K/\Q)$ with $\sigma(\zeta_m)=\zeta_m^b$,
    \[
    \inv_\p([\mathcal D])=b\inv_{\sigma(\p)}([\mathcal D])\in\Q/\Z.
    \]
\end{lemma}
In particular, we can confirm the second half of Conjecture~\ref{main-conjecture}~\eqref{symplectic-conjecture}:
\begin{cor}
Let $\pi$ be an irreducible self-dual supercuspidal representation of $\GL_n(F)$, satisfying Hypothesis~\ref{finite-order-hypo}. Then for any place $v$ of $\Q(\pi)$,
\[
\inv_v(\mathcal D_{\JL(\pi)}),\inv_v(\mathcal D_{\cL_F(\pi)})\in\frac12\Z/\Z,
\]
and the values (respectively) only depend on the characteristic of $v$. In particular, when $\Q(\pi)$ has an odd number of places above $p$ (in particular when $[\Q(\pi):\Q]$ is odd), for any $v|p$,
\begin{equation}\label{odd-equation}
\inv_v(\mathcal D_{\JL(\pi)})+\inv_v(\mathcal D_{\cL_F(\pi)})=\frac12[\Q(\pi):\Q].
\end{equation}
\end{cor}
\begin{proof}
Since $\pi$ is self-dual, $\Q(\pi)$ is totally real, so it only contains the root of unity $\zeta_2$. Thus the invariants must be $2$-torsion and dependent only on the characteristic of $v$ by Lemma~\ref{schur-subgroup-lemma}. 

In particular, if $\Q(\pi)$ has an odd number of places over $p$, equation~\eqref{sum-identity} is enough to pin down the invariants, and implies \eqref{odd-equation}.
\end{proof}

We also have uniform bounds on the local indices of the division algebras $\mathcal D_{\JL(\pi)}$ and $\mathcal D_{\cL_F(\pi)}$:
\begin{prop}
Let $\pi$ be a supercuspidal representation of $\GL_n(F)$ satisfying Hypothesis~\ref{finite-order-hypo}. Then for each prime $v|p$ of $\Q(\pi)$,
\[
\inv_v(\mathcal D_{\JL(\pi)}),\inv_v(\mathcal D_{\cL_F(\pi)})\in\frac1{(n,p-1)}\Z/\Z.
\]
In particular, if $(p-1,n)\le2$, Conjecture~\ref{main-conjecture} holds.
\end{prop}
\begin{proof}
By \cite{yamada}, both of $\inv_v(\mathcal D_{\JL(\pi)})$ and $\inv_v(\mathcal D_{\cL_F(\pi)})$ are $(p-1)$-torsion. Moreover, since $\inv_v(\mathcal D_{\cL_F(\pi)})$ is $n$-torsion, since $\dim(\cL_F(\pi))=n$. Thus $\inv_v(\mathcal D_{\cL_F(\pi)})$ is also torsion under the greatest common divisor $(n,p-1)$.

Moreover, $D^\times/F^\times$ is an extension of a pro-$p$ group by $\F_{q^n}^\times/\F_q$, which has order $1+q+\cdots+q^{n-1}$. Thus $\inv_v(\mathcal D_{\JL(\pi)})$ is also $p^N(1+q+\cdots+q^{n-1})$-torsion. Again $\inv_v(\mathcal D_{\JL(\pi)})$ is $(n,p-1)$-torsion.
\end{proof}

\bibliographystyle{amsalpha}
\bibliography{bibfile}

\end{document}